\documentclass[a4paper, 12pt, reqno]{amsart} 
\allowdisplaybreaks
\usepackage{amsmath,amssymb,amsthm}
\usepackage{microtype}
\usepackage[left=2.5cm, right=2.5cm, top=3cm, bottom=3cm]{geometry} 
\usepackage[colorlinks = true, urlcolor = blue, linkcolor = blue,
            citecolor = blue]{hyperref}
\usepackage{xcolor}
\usepackage[shortlabels]{enumitem}

\newtheorem{theorem}{Theorem}

\newtheorem{proposition}{Proposition}
\theoremstyle{definition}

\newtheorem{remark}{Remark}
\newtheorem{assumption}{Assumption}

\newcommand{\E}{\mathbb{E}}
\newcommand{\Prb}{\mathbb{P}}
\newcommand{\R}{\mathbb{R}}

\newcommand{\F}{\mathcal{F}}

\begin{document}
\title[Reinforced random walks with geometric inter-transition times]{Reinforced random walks with geometric inter-transition times}

\author[M. G. Coelho]{Mirela Gra\c{c}adio Coelho}
\author[F. P. A. Prado]{Fernando P. A. Prado}

\address[M. G. Coelho and F. P. A. Prado]{Departamento de Computa\c{c}\~ao e
  Matem\'atica, Universidade de S\~ao Paulo, 
  Avenida Bandeirantes 3900, Ribeir\~ao Preto, S\~ao Paulo, 
  14040-901, Brasil}
\email{mirela.coelho@usp.br, feprado@usp.br}

\date{\today}

\begin{abstract}
We consider interacting vertex-reinforced random walks on a finite graph, each transitioning according to independent geometric holding times of parameter $p_i \in (0,1]$. Letting $x=X(n)$ be the vector of vertex-occupation proportions up to time $n$, the one-step transition probabilities of walk $i$ are governed by $Q^i(x,p_i)=p_i\Pi^i(x)+(1-p_i)I$, where $\Pi^i(x)$ has rows equal to a probability measure $\pi^i(x)$ on the vertex set and $I$ is the identity. Its unique invariant measure is thus $\pi^i(x)$, independent of $p_i$. Consequently, the limiting points of $X(n)$ coincide with those of the simultaneous-transition model ($p_i=1$): the solutions of $x=\pi(x)$. However, almost sure convergence is non-trivial: the standard stochastic-approximation approach requires the Clark-Kushner condition, which is not immediate since the stochastic input is biased by the walk current state. We overcome this via a decomposition of the input into a martingale and a geometrically decaying correction, establishing almost sure convergence.
\end{abstract}

\maketitle

\section{Introduction}\label{sec:intro}

Reinforced stochastic processes form a broad class of models in which the probability of future transitions depends on the accumulated history of past events. The defining feature is a feedback mechanism: as certain states are visited more frequently, transitions to those states may become increasingly likely (positive reinforcement) or less likely (negative reinforcement). This memory dependence renders such processes inherently non-Markovian.

A prominent example is the vertex-reinforced random walk, introduced by Pemantle~\cite{P92} building on earlier work by Coppersmith and Diaconis~\cite{CD87} on edge-reinforced walks. In this model, a walker on a graph $G = (V, E)$ chooses its next vertex with probabilities that depend on how often each neighbouring vertex has been visited. One striking phenomenon is \emph{localisation}: under various conditions, the walk eventually confines itself to a strict subset of $V$, visiting only a random but fixed collection of vertices from some time onward. This behaviour has been established for several classes of graphs and weight functions; see~\cite{B97, PV99, V01, T04, BT11, BSS14, CT2017} and references therein.

While self-reinforced walks have been extensively studied, models involving \emph{multiple interacting} reinforced walks remain comparatively unexplored. Notable contributions include~\cite{C14, CDLM19, RPP22, PCR2023}, and more recently~\cite{PR2025}, which develops a general framework for interacting vertex-reinforced random walks on complete sub-graphs and establishes almost sure convergence of the joint vertex occupation measure to the fixed points of the transition map.

In the models cited above, all walks transition simultaneously at discrete times $n = 1, 2, 3, \ldots$. A natural question arises: what happens when different walks transition at different rates? In this article, we address this question by considering a model where each walk $i$ transitions at random times, with the waiting times between consecutive transitions being  independent geometric random variables with parameter $p_i \in (0,1]$. During a given time interval, one walk may transition more frequently than another, while the transition probability of each walk depends on the entire history of visits made by all walks to all vertices.

More precisely, let $W^i(n) \in \{1, \ldots, d\}$ denote the position of walk $i$ at time $n$ on a complete graph with $d \geq 2$ vertices, let $\xi^i_v(n) = \mathbf{1}\{W^i(n) = v\}$ be the indicator of a visit to vertex $v$, and let $X^i_v(n) = \frac{1}{n+1}\sum_{\ell=0}^{n} \xi^i_v(\ell)$ be the proportion of time that walk $i$ has spent at vertex $v$ up to time $n$. We collect these proportions into a vector $x = (X^1(n), \ldots, X^m(n)) \in \triangle^m$, where $\triangle$ denotes the $(d-1)$-simplex. Let $\pi^i_v: \triangle^m \to [0,1]$, $v = 1, \ldots, d$, $i = 1, \ldots, m$, be smooth functions satisfying $\sum_{v=1}^{d} \pi^i_v(x) = 1$ for all $x \in \triangle^m$. The transition probability of walk $i$ from vertex $w$ to vertex $v$, given the current proportion vector $x$, is
\[
q^i_{wv}(x, p_i) = 
\begin{cases}
p_i \pi^i_v(x) & \text{if } v \neq w,\\[4pt]
1 - p_i \sum_{k \neq w} \pi^i_k(x) & \text{if } v = w.
\end{cases}
\]
Using the identity $\sum_{k \neq w} \pi^i_k(x) = 1 - \pi^i_w(x)$, the transition matrix $Q^i(x, p_i) = (q^i_{wv}(x, p_i))_{w,v}$ can be written as
\[
Q^i(x, p_i) = p_i \Pi^i(x) + (1-p_i)I,
\]
where $\Pi^i(x)$ is the $d \times d$ matrix with all rows equal to $\pi^i(x) = (\pi^i_1(x), \ldots, \pi^i_d(x))$ and $I$ is the identity matrix. It follows immediately that $\pi^i(x)$ is the unique invariant measure of $Q^i(x, p_i)$, just as it is for $\Pi^i(x)$. Consequently, the equilibrium points of the system---the zeros of the vector field $F(x) = -x + \pi(x)$---are the same as in the simultaneous-transition model where $p_i = 1$ for all $i$.

However, verifying almost sure convergence to these equilibria is considerably more delicate. The standard approach via stochastic approximation relies on the Clark-Kushner condition~\eqref{eqn:clark_kushner}, which requires controlling the asymptotic behaviour of certain weighted sums of the stochastic input $U(n) = \xi(n) - \pi(X(n-1))$, where $\xi(n) = (\xi^1(n), \ldots, \xi^m(n))$ with $\xi^i(n) = (\xi^i_1(n), \ldots, \xi^i_d(n))$. In the simultaneous-transition case ($p_i = 1$), we have $\E[\xi(n+1)|\F_n] = \pi(X(n))$, so $U(n+1)$ is a martingale difference and convergence follows from standard martingale theory. In our setting, however,
\[
\E[\xi(n+1)|\F_n] = p\,\pi(X(n)) + (1-p)\,\xi(n),
\]
which is a convex combination of the target distribution and the current state. This introduces a bias that prevents direct application of the martingale argument.

Our main contribution is to overcome this obstacle by exploiting the recursive structure induced by the geometric waiting times. We show that the error term can be decomposed as
\[
U(n) = D(n) + (1-p)U(n-1) + E(n-1),
\]
where $D(n)$ is a martingale difference and $E(n-1)$ is a controlled correction. The factor $(1-p) < 1$ provides geometric decay, which converts potentially divergent sums into bounded series. This decomposition allows us to verify the Clark-Kushner condition and establish that the vertex occupation measure converges almost surely to the same equilibria as in the simultaneous-transition model.

The remainder of this article is organised as follows. Section~\ref{sec:model} introduces the model rigorously and presents the stochastic approximation framework. Section~\ref{sec:main} contains our main result and its proof. Section~\ref{sec:remarks} offers concluding remarks and directions for future work.

\section{Model and Preliminaries}\label{sec:model}

\subsection{Model definition}\label{subsec:model}

We present a generalisation of the models in~\cite{RPP22}, \cite{PCR2023} and~\cite{PR2025} to the case where $m \geq 1$ random walks transition from one vertex to another on a finite graph with $d \geq 2$ vertices at independent geometric inter-transition times with distinct parameters.

Given $m \geq 1$ independent sequences of independent and identically distributed geometric random variables $G^i_1, G^i_2, \ldots$ with parameter $p_i \in (0, 1]$, $i = 1, 2, \ldots, m$, we define the $k$-th transition time of walk $i$ as
\[
\tau^i_k = G^i_1 + G^i_2 + \cdots + G^i_k, \quad k = 1, 2, 3, \ldots
\]

We define the location of walk $i$ at time $n$ as a random vertex $W^i(n) \in \{1, 2, \ldots, d\}$ of the graph, where $\Prb(W^i(0) = v) = \alpha^i_v$, with $W^i(0)$, $i = 1, 2, \ldots, m$, being independent random variables and $\alpha = (\alpha^1, \alpha^2, \ldots, \alpha^m) \in \triangle^m$ a parameter of the model, where $\triangle$ denotes the $(d-1)$-dimensional simplex.

Regarding the dynamics of the walks, we postulate that $W^i(n)$ remains constant for $\tau^i_{k-1} < n \leq \tau^i_k$, with possible jumps only at times $n = \tau^i_k$, where $\tau^i_0 = 0$ and $k = 1, 2, \ldots$

For $n = \tau^i_k$, $k \geq 1$, the transition probability of walk $i$ to vertex $v \in \{1, 2, \ldots, d\}$ at its transition time $n = \tau^i_k$ is given by:
\begin{equation}\label{eqn:trans_prob}
\Prb\big(W^i(n+1) = v \mid X(n) = x\big) = \pi^i_v(x),
\end{equation}
where $\pi^i_v(x)$ is a continuous function assigning to each point $x \in \triangle^m$ a number $\pi^i_v(x) \in [0,1]$, and
\[
X(n) = \big(X^1_1(n), \ldots, X^1_d(n), \ldots, X^m_1(n), \ldots, X^m_d(n)\big),
\]
with $X^i_v(n)$ being the random proportion of time that walk $i$ spends at vertex $v$ up to time $n = 0, 1, 2, \ldots$

More precisely,
\begin{equation}\label{eqn:occupation}
X^i_v(n) = \frac{1}{n+1}\sum_{\ell=0}^{n} \xi^i_v(\ell), \quad \text{where} \quad \xi^i_v(\ell) = 
\begin{cases}
1 & \text{if } W^i(\ell) = v,\\
0 & \text{if } W^i(\ell) \neq v.
\end{cases}
\end{equation}

Thus, for any $n = 1, 2, \ldots$, the transition probability of $W^i(n)$ can be formulated as follows:
\begin{equation}\label{eqn:trans_matrix_entry}
\Prb\big(W^i(n+1) = v \mid W^i(n) = w, X(n) = x\big) = q^i_{wv}(x, p_i),
\end{equation}
where
\begin{equation}\label{eqn:q_def}
q^i_{wv}(x, p_i) = 
\begin{cases}
p_i \pi^i_v(x) & \text{if } v \neq w,\\[4pt]
1 - p_i \displaystyle\sum_{k \neq w} \pi^i_k(x) & \text{if } v = w.
\end{cases}
\end{equation}

\subsection{Invariant measure}\label{subsec:invariant}

Let $Q^i(x, p_i)$ be the $d \times d$ matrix whose entry in row $w$ and column $v$ is given by $q^i_{wv}(x, p_i)$ as defined in~\eqref{eqn:q_def}. Using the identity $\sum_{k \neq w} \pi^i_k(x) = 1 - \pi^i_w(x)$, we can write
\begin{equation}\label{eqn:Q_structure}
Q^i(x, p_i) = p_i \Pi^i(x) + (1-p_i)I,
\end{equation}
where $\Pi^i(x)$ is the $d \times d$ matrix with all rows equal to $\pi^i(x) = (\pi^i_1(x), \ldots, \pi^i_d(x))$ and $I$ is the identity matrix. Since $\pi^i(x) \Pi^i(x) = \pi^i(x)$, it follows that 
 $\pi^i(x)$ is the unique invariant measure of $Q^i(x, p_i)$, independently of $p_i$. In particular, the models in~\cite{RPP22}, \cite{PCR2023} and~\cite{PR2025} correspond to the special case $p_i = 1$ for all $i$.

\subsection{Stochastic approximation framework}\label{subsec:stoch_approx}

Let $F: \triangle^m \to \R^{dm}$ be the function given by
\begin{equation}\label{eqn:vector_field}
F(x) = -x + \pi(x),
\end{equation}
where $\pi: \triangle^m \to \triangle^m$ is defined by $\pi(x) = (\pi^1(x), \pi^2(x), \ldots, \pi^m(x))$, with $\pi^i(x) = (\pi^i_1(x), \pi^i_2(x), \ldots, \pi^i_d(x))$.

\begin{assumption}\label{ass:lipschitz}
We assume that $F: \triangle^m \to \R^{md}$ is a Lipschitz continuous vector field.
\end{assumption}

\begin{assumption}\label{ass:lyapunov}
We assume that there exists a strict Lyapunov function $L$ for the forward invariant set $F^{-1}(0)$ of the semi-flow induced by the vector field $F$ defined in~\eqref{eqn:vector_field}, such that $L(F^{-1}(0))$ is finite.
\end{assumption}

Let $\F_n$ be the sigma-algebra generated by $W(0), W(1), \ldots, W(n)$, $n = 0, 1, \ldots$

\begin{remark}\label{rmk:norm}
Throughout this exposition, we fix on $\R^d$ the norm $\|v\| := \sum_{j=1}^{d} |v_j|$ and for $d \times d$ matrices the induced operator norm: $\|A\| := \max_i \sum_j |A_{ij}|$ (maximum of the absolute row sums).
\end{remark}

\begin{theorem}\label{thm:stoch_approx}
Suppose that the difference $X(n+1) - X(n)$ can be written as
\begin{equation}\label{eqn:recursion}
X(n+1) - X(n) = \gamma_n \big(F(X(n)) + U(n+1)\big),
\end{equation}
where $\gamma_n \in \R_+$ and $U(n+1) \in \R^{md}$ are such that $\lim_{n \to \infty} \gamma_n = 0$, $\sum_{n \geq 0} \gamma_n = \infty$, $\sum_{n \geq 0} \gamma_n^2 < \infty$, and $U(n+1) \in \F_{n+1}$.

If, for $\tau_0 = 0$, $\tau_n = \sum_{k=0}^{n-1} \gamma_k$, and $T > 0$,
\begin{equation}\label{eqn:clark_kushner}
\lim_{n \to \infty} \left(\sup_{r:\, 0 \leq \tau_r - \tau_n \leq T} \left\|\sum_{k=n}^{r-1} U(k) \gamma_k\right\|\right) = 0 \quad \text{almost surely},
\end{equation}
then, under Assumptions~\ref{ass:lipschitz} and~\ref{ass:lyapunov}, the set of accumulation points of $X(n)$, denoted by $L(X)$, is connected and satisfies $L(X) \subseteq F^{-1}(0)$. In particular, if $F^{-1}(0)$ is countable, then $X(n)$ converges to some point of $F^{-1}(0)$ almost surely.
\end{theorem}

The proof of Theorem~\ref{thm:stoch_approx} follows from the dynamical systems approach to stochastic approximations as presented in Bena{\"i}m~\cite{B96, B99}.

\begin{proposition}\label{prop:recursion}
The process $X$ satisfies condition~\eqref{eqn:recursion} of Theorem~\ref{thm:stoch_approx}.
\end{proposition}

\begin{proof}
To obtain~\eqref{eqn:recursion}, we define
\begin{equation}\label{eqn:gamma_U_def}
\gamma_n = \frac{1}{n+2} \quad \text{and} \quad U^i_v(n+1) = \xi^i_v(n+1) - \pi^i_v(X(n)).
\end{equation}
Note that $U^i_v(n+1)$ and $\xi^i_v(n+1)$ are adapted to $\F_{n+1}$. Taking into account that $F^i_v(x) = -x^i_v + \pi^i_v(x)$, and substituting~\eqref{eqn:gamma_U_def} into the right-hand side of~\eqref{eqn:recursion}, we obtain, component by component,
\[
\gamma_n \big(F^i_v(X(n)) + U^i_v(n+1)\big) = \frac{1}{n+2}\big(\xi^i_v(n+1) - X^i_v(n)\big).
\]
To conclude~\eqref{eqn:recursion}, it remains to prove that $X^i_v(n+1) - X^i_v(n)$ equals the right-hand side above. Taking into account the definition of $X^i_v(n)$ in~\eqref{eqn:occupation}, we have $X^i_v(n+1) = \big((n+1)X^i_v(n) + \xi^i_v(n+1)\big)/(n+2)$. Therefore,
\[
X^i_v(n+1) - X^i_v(n) = \frac{1}{n+2}\big(\xi^i_v(n+1) - X^i_v(n)\big).
\]
Thus we have shown that, for $\gamma_n$ and $U(n+1)$ defined by~\eqref{eqn:gamma_U_def}, equation~\eqref{eqn:recursion} holds, with $U(n+1) \in \F_{n+1}$.
\end{proof}

\section{Main Result}\label{sec:main}

Our goal is to identify the convergence points of the process $X$ as the zeros of $F$, by applying Theorem~\ref{thm:stoch_approx}. The main obstacle is that the characterisation of $X$ as a stochastic approximation faces a difficulty: $\pi(X(n)) \neq \E[\xi(n+1)|\F_n]$. Indeed, $\E[\xi(n+1)|\F_n]$ depends not only on $X(n)$ but also on the states visited by the walks at the immediately preceding time $n$. If we had $\pi(X(n)) = \E[\xi(n+1)|\F_n]$, then $U(n+1)\gamma_{n+1} = (\xi(n+1) - \E[\xi(n+1)|\F_n])\gamma_{n+1}$ would be a martingale difference, from which condition~\eqref{eqn:clark_kushner} would follow; see~\cite{RPP22}. The main challenge of this work consists precisely in proving~\eqref{eqn:clark_kushner} in the present case where $U(n+1)\gamma_{n+1} \neq (\xi(n+1) - \E[\xi(n+1)|\F_n])\gamma_{n+1}$.

\begin{theorem}\label{thm:main}
Let $\gamma_n > 0$ and $U(n+1) \in \F_{n+1}$ be as introduced in~\eqref{eqn:gamma_U_def}. Then condition~\eqref{eqn:clark_kushner} holds.
\end{theorem}

\subsection{Proof of Theorem~\ref{thm:main}}\label{subsec:proof}


%



To lighten the notation, we assume without loss of generality that $m = 1$ and omit the index $i$. Thus, we write $X(n)$, $\xi(n)$, $W(n)$, $\pi(x)$, $Q(x, p)$, $\Pi(x)$ instead of $X^1(n)$, $\xi^1(n)$, $W^1(n)$, $\pi^1(x)$, $Q^1(x, p_1)$, $\Pi^1(x)$, respectively.

\medskip
\noindent\textbf{Central idea of the proof.} First, we observe that a direct estimate of the sum $\sum_{k=n}^{r-1} U(k)\gamma_k$ by the triangle inequality would yield $\sum_{k=n}^{r-1} \|U(k)\|\gamma_k \leq 2\sum_{k=n}^{r-1} \gamma_k$, which diverges since $\sum \gamma_k = \infty$. The key to the proof is to exploit the special structure of the transition matrix $Q(x, p) = p\Pi(x) + (1-p)I$, the Lipschitz continuity of $F(x)$, and the representation $X(n+1) - X(n)$ in~\eqref{eqn:recursion} to deduce the recursive form of $U(n)$ in~\eqref{eqn:U_recursion} and decompose $U(n)$ into three parts: a martingale difference $D(n)$, a term proportional to $U(n-1)$ with factor $(1-p) < 1$, and a small correction term. The weighted martingale differences $\mathcal{D}(n) = \gamma_{n-1} D(n)$ are square-summable, so their sum $\mathcal{M}(n) = \sum_{j=1}^{n} \mathcal{D}(j)$ converges. This allows exploiting cancellations that overcome the divergent condition $\sum \gamma_k = \infty$. The remaining terms are controlled by the geometric decay $(1-p)^k$, which converts potentially divergent sums into bounded geometric series.

\medskip

The proof is organised in several steps.

\medskip
\noindent\textbf{Step 1: Decomposition of $U(n)$.}

Recall that $U(n) = \xi(n) - \pi(X(n-1))$ for $n \geq 1$, where $\xi(n) = (\xi_1(n), \ldots, \xi_d(n))$ is the indicator vector of the state visited at time $n$, i.e., $\xi_v(n) = 1$ if $W(n) = v$ and $\xi_v(n) = 0$ otherwise.

Adding and subtracting $\E[\xi(n)|\F_{n-1}]$ from $U(n)$, we obtain the decomposition:
\begin{equation}\label{eqn:U_decomp}
U(n) = D(n) + R(n-1),
\end{equation}
where $D(n) = \xi(n) - \E[\xi(n)|\F_{n-1}]$ is a martingale difference, i.e., $\E[D(n)|\F_{n-1}] = 0$, and $R(n-1) = \E[\xi(n)|\F_{n-1}] - \pi(X(n-1))$ is the bias term, which is $\F_{n-1}$-measurable.

\medskip
\noindent\textbf{Step 2: Calculation of $\E[\xi(n)|\F_{n-1}]$.}

Given $\F_{n-1}$, we know $W(n-1)$ and $X(n-1)$. The transition is governed by the matrix $Q(X(n-1), p)$. For each component $v \in \{1, \ldots, d\}$:
\[
\E[\xi_v(n)|\F_{n-1}] = \Prb(W(n) = v|\F_{n-1}) = q_{W(n-1),v}(X(n-1), p).
\]
In vector notation, using the fact that $\xi(n-1)$ is the canonical vector $e_{W(n-1)}$ (the vector with 1 in position $W(n-1)$ and 0 elsewhere):
\[
\E[\xi(n)|\F_{n-1}] = \xi(n-1) Q(X(n-1), p).
\]
Applying the structure~\eqref{eqn:Q_structure}, i.e., $Q(x, p) = p\Pi(x) + (1-p)I$:
\begin{align*}
\xi(n-1) Q(X(n-1), p) &= \xi(n-1)\big[p\Pi(X(n-1)) + (1-p)I\big]\\
&= p\,\pi(X(n-1)) + (1-p)\xi(n-1).
\end{align*}
The last equality follows from the fact that, for any row vector $\lambda$ (in particular for $\xi(n-1) = e_{W(n-1)}$), we have $\lambda\Pi(x) = \pi(x)$, since all rows of $\Pi(x)$ are equal to $\pi(x)$.

Therefore:
\begin{equation}\label{eqn:conditional_exp}
\E[\xi(n)|\F_{n-1}] = p\,\pi(X(n-1)) + (1-p)\xi(n-1).
\end{equation}

\medskip
\noindent\textbf{Step 3: Expression for $R(n-1)$.}

Substituting~\eqref{eqn:conditional_exp} into the definition of $R(n-1)$:
\begin{align*}
R(n-1) &= \E[\xi(n)|\F_{n-1}] - \pi(X(n-1))\\
&= (1-p)\big(\xi(n-1) - \pi(X(n-1))\big).
\end{align*}
Recall that $U(n-1) = \xi(n-1) - \pi(X(n-2))$. Thus:
\[
\xi(n-1) - \pi(X(n-1)) = U(n-1) + \pi(X(n-2)) - \pi(X(n-1)) = U(n-1) + \Delta_{n-1},
\]
where we define the correction term:
\begin{equation}\label{eqn:Delta_def}
\Delta_{n-1} := \pi(X(n-2)) - \pi(X(n-1)).
\end{equation}
Therefore:
\begin{equation}\label{eqn:R_expression}
R(n-1) = (1-p)U(n-1) + (1-p)\Delta_{n-1}.
\end{equation}

\medskip
\noindent\textbf{Estimate of $\|\Delta_{n-1}\|$.} Since $\pi$ is Lipschitz continuous with constant $L > 0$ and $X(n-1) - X(n-2) = \gamma_{n-2}(F(X(n-2)) + U(n-1))$:
\[
\|\Delta_{n-1}\| = \|\pi(X(n-2)) - \pi(X(n-1))\| \leq L\|X(n-2) - X(n-1)\| = L\gamma_{n-2}\|F(X(n-2)) + U(n-1)\|.
\]
Since $\|F(x)\| \leq 2$ and $\|U(n-1)\| \leq 2$ (because $\xi, \pi \in \triangle$), we obtain:
\begin{equation}\label{eqn:Delta_bound}
\|\Delta_{n-1}\| \leq 4L\gamma_{n-2}.
\end{equation}

\medskip
\noindent\textbf{Step 4: Recurrence for $U(n)$.}

From the decomposition~\eqref{eqn:U_decomp} and the expression~\eqref{eqn:R_expression}:
\begin{equation}\label{eqn:U_recursion}
U(n) = D(n) + (1-p)U(n-1) + E(n-1),
\end{equation}
where $E(n-1) = (1-p)\Delta_{n-1}$ is the error term, satisfying:
\begin{equation}\label{eqn:E_bound}
\|E(n-1)\| \leq 4L(1-p)\gamma_{n-2}.
\end{equation}

\medskip
\noindent\textbf{Step 5: Solution of the recurrence.}

Iterating the recurrence~\eqref{eqn:U_recursion} from an initial index $n_0$ to $k$, we obtain:
\begin{equation}\label{eqn:U_solution}
U(k) = \sum_{j=n_0+1}^{k} (1-p)^{k-j} D(j) + (1-p)^{k-n_0} U(n_0) + \sum_{j=n_0}^{k-1} (1-p)^{k-1-j} E(j).
\end{equation}
We name the three terms:
\begin{align}
A(n_0, k) &= \sum_{j=n_0+1}^{k} (1-p)^{k-j} D(j) \quad \text{(martingale term)}, \label{eqn:A_def}\\
B(n_0, k) &= (1-p)^{k-n_0} U(n_0) \quad \text{(initial condition)}, \label{eqn:B_def}\\
C(n_0, k) &= \sum_{j=n_0}^{k-1} (1-p)^{k-1-j} E(j) \quad \text{(correction term)}. \label{eqn:C_def}
\end{align}

\medskip
\noindent\textbf{Step 6: Analysis of the sum $\sum_{k=n}^{r-1} U(k)\gamma_k$.}

Using the expansion~\eqref{eqn:U_solution} with $n_0 = n$:
\begin{equation}\label{eqn:sum_decomp}
\sum_{k=n}^{r-1} U(k)\gamma_k = \underbrace{\sum_{k=n}^{r-1} \gamma_k A(n, k)}_{(\mathrm{I})} + \underbrace{\sum_{k=n}^{r-1} \gamma_k B(n, k)}_{(\mathrm{II})} + \underbrace{\sum_{k=n}^{r-1} \gamma_k C(n, k)}_{(\mathrm{III})}.
\end{equation}
We analyse each term separately.

\medskip
\noindent\textbf{Step 7: Analysis of Term (II) --- Initial condition.}

We have:
\[
(\mathrm{II}) = \sum_{k=n}^{r-1} \gamma_k (1-p)^{k-n} U(n) = U(n) \sum_{k=n}^{r-1} \gamma_k (1-p)^{k-n}.
\]
Since $\gamma_k$ is decreasing and $\sum_{\ell=0}^{\infty}(1-p)^\ell = 1/p$:
\[
\sum_{k=n}^{r-1} \gamma_k (1-p)^{k-n} \leq \frac{\gamma_n}{p}.
\]
Since $\|U(n)\| \leq 2$:
\begin{equation}\label{eqn:term_II}
\|(\mathrm{II})\| \leq \frac{2\gamma_n}{p} \xrightarrow{n \to \infty} 0, \quad \text{uniformly in } r.
\end{equation}

\medskip
\noindent\textbf{Step 8: Analysis of Term (III) --- Correction.}

We have:
\[
(\mathrm{III}) = \sum_{k=n}^{r-1} \gamma_k \sum_{j=n}^{k-1} (1-p)^{k-1-j} E(j).
\]
Exchanging the order of summation (for each fixed $j$, $k$ ranges from $j+1$ to $r-1$):
\[
(\mathrm{III}) = \sum_{j=n}^{r-2} E(j) \sum_{k=j+1}^{r-1} \gamma_k (1-p)^{k-1-j}.
\]
Since $\gamma_k$ is decreasing, the inner sum satisfies $\sum_{k=j+1}^{r-1} \gamma_k (1-p)^{k-1-j} \leq \gamma_{j+1}/p$.
Therefore:
\[
\|(\mathrm{III})\| \leq \sum_{j=n}^{r-2} \|E(j)\| \cdot \frac{\gamma_{j+1}}{p} \leq \sum_{j=n}^{r-2} 4L(1-p)\gamma_{j-1} \cdot \frac{\gamma_{j+1}}{p}.
\]
Since $\gamma_{j-1}\gamma_{j+1} \leq 2\gamma_j^2$:
\begin{equation}\label{eqn:term_III}
\|(\mathrm{III})\| \leq \frac{8L(1-p)}{p} \sum_{j=n}^{\infty} \gamma_j^2 \xrightarrow{n \to \infty} 0, \quad \text{uniformly in } r.
\end{equation}

\medskip
\noindent\textbf{Step 9: Analysis of Term (I) --- Martingale.}

We have:
\[
(\mathrm{I}) = \sum_{k=n}^{r-1} \gamma_k \sum_{j=n+1}^{k} (1-p)^{k-j} D(j).
\]
Exchanging the order of summation (for each fixed $j$, $k$ ranges from $j$ to $r-1$):
\[
(\mathrm{I}) = \sum_{j=n+1}^{r-1} D(j) \sum_{k=j}^{r-1} \gamma_k (1-p)^{k-j}.
\]
We define $c^r_j = \sum_{k=j}^{r-1} \gamma_k (1-p)^{k-j}$, which satisfies $c^r_j \leq \gamma_j/p$.

\medskip
\noindent\textbf{Decomposition.} We write:
\[
c^r_j = \frac{\gamma_j}{p} - d^r_j,
\]
where $d^r_j = \frac{\gamma_j}{p} - c^r_j \geq 0$.

Then:
\[
(\mathrm{I}) = \frac{1}{p} \sum_{j=n+1}^{r-1} \gamma_j D(j) - \sum_{j=n+1}^{r-1} d^r_j D(j).
\]

\medskip
\noindent\textbf{Analysis of the first term.} Define the weighted martingale difference $\mathcal{D}(j) = \gamma_{j-1} D(j)$. Observe that $\mathcal{D}(j)$ is a martingale difference, since $\gamma_{j-1}$ is deterministic and $\E[D(j)|\F_{j-1}] = 0$. Since $\|D(j)\| \leq 2$, we have $\E[\|\mathcal{D}(j)\|^2 | \F_{j-1}] \leq 4\gamma_{j-1}^2$. Therefore:
\[
\sum_{j=1}^{\infty} \E[\|\mathcal{D}(j)\|^2 | \F_{j-1}] \leq 4 \sum_{j=1}^{\infty} \gamma_{j-1}^2 < \infty.
\]
By the martingale convergence theorem (see, e.g., Durrett~\cite{D19}), the martingale $\mathcal{M}(n) = \sum_{j=1}^{n} \mathcal{D}(j) = \sum_{j=1}^{n} \gamma_{j-1} D(j)$ converges almost surely. Therefore, $\mathcal{M}(n)$ is a Cauchy sequence almost surely, and the difference $\mathcal{M}(r-1) - \mathcal{M}(n)$, viewed as a sequence of functions indexed by $n$ with argument $r-1 \geq n$, converges uniformly (in $r$) to zero as $n \to \infty$. That is:
\begin{equation}\label{eqn:martingale_tail}
\left\|\sum_{j=n+1}^{r-1} \gamma_{j-1} D(j)\right\| = \|\mathcal{M}(r-1) - \mathcal{M}(n)\| \xrightarrow{n \to \infty} 0 \quad \text{a.s., uniformly in } r.
\end{equation}
Since $\gamma_j \leq 2\gamma_{j-1}$ for all $j \geq 1$, we also have
\begin{equation}\label{eqn:martingale_tail_gamma_j}
\left\|\sum_{j=n+1}^{r-1} \gamma_j D(j)\right\| \xrightarrow{n \to \infty} 0 \quad \text{a.s., uniformly in } r.
\end{equation}

\medskip
\noindent\textbf{Analysis of the second term.} We have $d^r_j = \frac{\gamma_j}{p} - c^r_j$. Writing explicitly:
\[
d^r_j = \sum_{\ell=0}^{\infty} \gamma_j (1-p)^\ell - \sum_{\ell=0}^{r-1-j} \gamma_{j+\ell} (1-p)^\ell = \underbrace{\sum_{\ell=0}^{r-1-j} (\gamma_j - \gamma_{j+\ell})(1-p)^\ell}_{(a)} + \underbrace{\gamma_j \sum_{\ell=r-j}^{\infty} (1-p)^\ell}_{(b)}.
\]

For term $(b)$: $(b) = \frac{\gamma_j}{p}(1-p)^{r-j}$.

For term $(a)$: Since $\gamma_k = \frac{1}{k+2}$, we have $\gamma_j - \gamma_{j+\ell} = \frac{\ell}{(j+2)(j+\ell+2)} \leq \frac{\ell}{(j+2)^2}$. Hence:
\[
(a) \leq \frac{1}{(j+2)^2} \sum_{\ell=0}^{\infty} \ell (1-p)^\ell = \frac{1-p}{p^2(j+2)^2}.
\]

Therefore:
\[
d^r_j \leq \frac{1-p}{p^2(j+2)^2} + \frac{\gamma_j}{p}(1-p)^{r-j}.
\]

Thus:
\[
\left\|\sum_{j=n+1}^{r-1} d^r_j D(j)\right\| \leq 2 \sum_{j=n+1}^{\infty} \left[\frac{1-p}{p^2(j+2)^2} + \frac{\gamma_j}{p}(1-p)^{r-j}\right].
\]

The sum on the right-hand side of the inequality immediately above splits into two sums. Since $\gamma_n = 1/(n+2)$, both sums are $O(1/n)$ and go to zero as $n \to \infty$.

Hence:
\begin{equation}\label{eqn:d_bound}
\left\|\sum_{j=n+1}^{r-1} d^r_j D(j)\right\| \xrightarrow{n \to \infty} 0, \quad \text{uniformly in } r.
\end{equation}

Combining~\eqref{eqn:martingale_tail_gamma_j} and~\eqref{eqn:d_bound}:
\begin{equation}\label{eqn:term_I}
\sup_{r \geq n} \|(\mathrm{I})\| \xrightarrow{n \to \infty} 0 \quad \text{a.s.}
\end{equation}

\medskip
\noindent\textbf{Step 10: Conclusion.}

From~\eqref{eqn:sum_decomp},~\eqref{eqn:term_II},~\eqref{eqn:term_III}, and~\eqref{eqn:term_I}:
\[
\sup_{r \geq n} \left\|\sum_{k=n}^{r-1} U(k)\gamma_k\right\| \leq \|(\mathrm{I})\| + \|(\mathrm{II})\| + \|(\mathrm{III})\| \xrightarrow{n \to \infty} 0 \quad \text{a.s.}
\]
In particular, for any $T > 0$, since the condition $0 \leq \tau_r - \tau_n \leq T$ implies $r \geq n$:
\[
\lim_{n \to \infty} \sup_{r:\, 0 \leq \tau_r - \tau_n \leq T} \left\|\sum_{k=n}^{r-1} U(k)\gamma_k\right\| = 0 \quad \text{a.s.}
\]
This completes the proof of~\eqref{eqn:clark_kushner}. \qed

\section{Concluding Remarks}\label{sec:remarks}

In this work, we generalised the models of interacting vertex-reinforced random walks presented in~\cite{RPP22}, \cite{PCR2023} and~\cite{PR2025}, allowing each walk $i$ to transition at independent geometric times with distinct parameters $p_i \in (0,1]$.

In the original model, all walks transition simultaneously at discrete times $n = 1, 2, 3, \ldots$, and the transition probability of walk $i$ to vertex $v$ is given directly by $\pi^i_v(x)$, depending only on the proportion vector $x$ of visits by all walks to all vertices, independently of the position occupied by the walk at the preceding time. In the present work, we introduced a generalisation where each walk $i$ transitions only at its random transition times $\tau^i_1, \tau^i_2, \tau^i_3, \ldots$, remaining stationary at all other times. This modification implies that the transition matrix $Q^i(x, p_i)$ depends explicitly on the parameter $p_i$, as defined in~\eqref{eqn:q_def}.

The central result of this work, established in Theorem~\ref{thm:main}, demonstrates that the accumulation points of the process $X$ are contained in the zeros of the vector field $F(x) = -x + \pi(x)$. A notable aspect of the results obtained is that, although the dynamics of the process depends on the parameters $p_i$ of the geometric transition times, the invariant measure $\pi^i(x)$ of the matrix $Q^i(x, p_i)$ does not depend on $p_i$. Consequently, the accumulation points of the generalised process coincide with those of the original model where $p_i = 1$ for all $i$. In other words, the introduction of independent geometric transition times with distinct parameters does not alter the limiting points of process $X$.

Several relevant questions remain open. Although the accumulation points are the same, the probability distribution of the limiting vertex occupation proportions may be substantially affected by the parameters $p_i$. Investigating how the different parameters $p_i$ influence this distribution constitutes an interesting direction for future work. In particular, walks with larger parameters $p_i$ transition more frequently and may, in competitive dynamics, occupy more attractive vertices with higher probability, excluding opponent walks more effectively. Another point to be explored consists of identifying convergence points of similar processes where reinforced walks transition on non-complete graphs. The results obtained in this work suggest that such investigation may be approached in a similar manner, provided the transition probability matrices satisfy a uniform contractivity condition as a function of $x$. One natural candidate is a uniform bound on the Dobrushin coefficient of each matrix $Q^i(x, p_i)$, which controls the rate of convergence to the invariant measure and may provide the necessary regularity for extending the present arguments.


\bibliographystyle{amsplain}

\end{document}